\documentclass[11pt]{amsart}

\pdfoutput=1

\usepackage[foot]{amsaddr}

\usepackage{cancel}
\usepackage{esint,amssymb} 
\usepackage{graphicx}
\usepackage{MnSymbol}
\usepackage{mathtools} 
\usepackage[colorlinks=true, pdfstartview=FitV, linkcolor=blue, citecolor=blue, urlcolor=blue, pagebackref=false]{hyperref}
\usepackage{orcidlink}
\usepackage{microtype}

\usepackage{bm}
\usepackage{dsfont}
\usepackage{mathrsfs}

\definecolor{darkgreen}{rgb}{0,0.5,0}
\definecolor{darkblue}{rgb}{0,0,0.7}
\definecolor{darkred}{rgb}{0.9,0.1,0.1}

\newtheorem{proposition}{Proposition}
\newtheorem{theorem}[proposition]{Theorem}
\newtheorem{lemma}[proposition]{Lemma}

\theoremstyle{remark}

\newtheorem{example}[proposition]{Example}
\theoremstyle{definition}
\newtheorem{definition}[proposition]{Definition}

\numberwithin{equation}{section}
\numberwithin{proposition}{section}
\numberwithin{figure}{section}
\numberwithin{table}{section}

\newcommand{\N}{\mathbb{N}}

\newcommand{\R}{\mathbb{R}}

\renewcommand{\leq}{\leqslant}
\renewcommand{\geq}{\geqslant}

\renewcommand{\subset}{\subseteq}

\renewcommand{\d}{\mathrm{d}}

\newcommand{\mcl}{\mathcal}
\newcommand{\msf}{\mathsf}
\newcommand{\mfk}{\mathfrak}

\newenvironment{e}{\begin{equation}}{\end{equation}\ignorespacesafterend}
\newenvironment{e*}{\begin{equation*}}{\end{equation*}\ignorespacesafterend}

\begin{document}

\author{Victor Issa\,\orcidlink{0009-0009-1304-046X}}
\address[Victor Issa]{Department of Mathematics, ENS de Lyon, Lyon, France}
\email{\href{mailto:victor.issa@ens-lyon.fr}{victor.issa@ens-lyon.fr}}

\title[Weak-Strong Uniqueness Principle]{Weak-Strong Uniqueness Principle for Hamilton-Jacobi Equations}

\begin{abstract}
   We show that if a Hamilton-Jacobi equation admits a differentiable solution whose gradient is Lipschitz, then this solution is the unique semi-concave weak solution. Our result does not rely on any convexity (nor concavity) assumptions on the initial condition or the nonlinearity, and can therefore be utilized in contexts where the viscosity solution admits no standard variational representation.

   \bigskip

    \noindent \textsc{Keywords and phrases: Hamilton-Jacobi equations, viscosity solutions, weak-strong uniqueness}  

    \medskip

    \noindent \textsc{MSC 2020: 35F21, 35D30, 35D40} 
\end{abstract}

\maketitle

\newpage 

\newpage
\pagenumbering{arabic}
\section{Introduction}

We let $H : \R^D \to \R$ be a $\mcl C^2$ function and $\psi : \R^D \to \R$ be a Lipschitz and differentiable function whose gradient is Lipschitz. We are interested in partial differential equations of the form
\begin{e} \label{e.HJ}
    \begin{cases}
        \partial_t g - H(\nabla g) = 0 &\text{ on } (0,+\infty) \times \R^D \\
        g(0,\cdot) = \psi              &\text{ on } \R^D.
    \end{cases}
\end{e}
The equation \eqref{e.HJ} may not have any differentiable solution \cite[Section~3.2~Examples~5~\&~6]{evans}. But, by introducing the notion of viscosity solution it is possible to guarantee that \eqref{e.HJ} admits exactly one solution \cite[Section~10]{evans}. In the literature, many efforts have been devoted to identifying contexts in which the notion of viscosity solution coincides with other weaker notions of solution. It is for example well-known that when $H$ is convex, any semi-convex function which solves \eqref{e.HJ} almost everywhere is in fact equal to the viscosity solution \cite[Section~3.2 Theorem~7]{evans}.  It has also been recently proven that when $H$ is convex and the initial condition is regular enough, any semi-concave function which solves \eqref{e.HJ} almost everywhere is equal to the viscosity solution \cite{issa2024uniqueness}.

A particularly striking, but for the moment hypothetical, application of this type of result would be the identification of the limit free energy of some mean-field spin glass models via partial differential equations arguments. In this context, one studies the limiting value of a quantity associated to a fully connected and random ensemble of spins, called the limit free energy \cite{chenmourrat2023cavity,HJbook,pan}. The limit free energy depends on a time parameter $t \in \R_+$ and a space parameter $\msf q$ which belongs to a closed convex cone embedded in $L^2([0,1),\R^D)$. It is known that the limit free energy is a semi-concave function \cite[Proposition~3.8]{chenmourrat2023cavity} which satisfies an equation of the form \eqref{e.HJ} “almost everywhere” (outside a Gaussian null set) \cite[Theorem~2.6 \& Proposition~7.2]{chenmourrat2023cavity}. In this sense, the limit free energy is a semi-concave weak solution of an equation of the form \eqref{e.HJ} on a subset of an Hilbert space. To understand the behavior of the limit free energy, one wishes to assert that it is in fact the unique viscosity solution of the corresponding equation \cite[Conjecture~2.6]{mourrat2019parisi}. For models whose free energy solves an equation of the form \eqref{e.HJ} with a convex nonlinearity, the limit free energy has already been identified via probabilistic methods \cite{chenmourrat2023cavity, gue03, mourrat2020extending, pan, Tpaper}. Thus, the truly interesting regime for this kind of application is the one where $H$ is not assumed to be convex (nor concave). This highlights the need of establishing selection principles with no convexity assumptions on the nonlinearity.

One such a selection principle has been identified in \cite{chen2022statistical} and fruitfully used to identify the limit free energy of a family of statistical inference problems. But this result still relies on convexity assumptions for the initial condition $\psi$ which are not satisfied by the corresponding initial condition appearing in the context of spin glasses. In this paper, we show that when the viscosity solution of \eqref{e.HJ} is regular enough, \eqref{e.HJ} admits exactly one semi-concave weak solution. This establishes a selection principle which holds regardless of the convexity properties of $H$ and $\psi$.

We equip $\R^D$ with the standard scalar product $x \cdot y = \sum_{d = 1}^D x_d y_d$, and we denote by $|\cdot|$ the associated norm. 
\begin{theorem} \label{t.mainresult}
    We recall that $H : \R^D \to \R$ is a $\mcl C^2$ function and $\psi : \R^D \to \R$ is a Lipschitz differentiable function whose gradient is Lipschitz. Let $T > 0$, we  assume that the unique viscosity solution $g$ of \eqref{e.HJ} is differentiable on $(0,T) \times \R^D$, and that there exists $L \geq 0$ such that for every $t \in [0,T)$ and $x,y \in \R^D$,
    \begin{e*}
        |\nabla g(t,x) - \nabla g(t,y)| \leq L|x-y|. 
    \end{e*}
    In this case, there is a unique Lipschitz function $f : [0,T] \times \R^D \to \R$ satisfying the following properties.
    \begin{enumerate}
        \item \label{i.initial} $f(0,\cdot) = \psi.$
        \item \label{i.aesol}For almost every $(t,x) \in (0,T) \times \R^D$, $f$ is differentiable at $(t,x)$ and 
        \begin{e*}
        \partial_t f(t,x) - H(\nabla f(t,x)) = 0.
        \end{e*}
        \item \label{i.semiconcave} There exists $c > 0$ such that for every $t \in [0,T)$, $x \mapsto \frac{c}{2}|x|^2 - f(t,x)$ is convex.
    \end{enumerate}
   In addition, the unique Lipschitz function $f$ described above is the viscosity solution $g$.
\end{theorem}
%
%
In condition \eqref{i.semiconcave} in Theorem~\ref{t.mainresult}, if the function $x \mapsto \frac{c}{2}|x|^2 - f(t,x)$ is replaced by the function $x \mapsto \frac{c}{2}|x|^2 + f(t,x)$ then the statement of Theorem~\ref{t.mainresult} remains true. This is because we can replace $(\psi,H)$ by $(-\psi, p\mapsto -H(-p))$ without affecting the hypotheses on $\psi$ and $H$.

We point out that when the nonlinearity $H$ is strongly convex (meaning that there exists $\theta > 0$ such that $p \mapsto H(p) - \theta |p|^2$ is convex), if the solution $g$ of \eqref{e.HJ} is differentiable on $(0,T') \times \R^D$, then for every $\varepsilon > 0$, $g$ satisfies the regularity assumption of Theorem~\ref{t.mainresult} with $T = T'-\varepsilon$ \cite[Theorem~15.1]{lions1982generalized}.

As shown in \cite[Section~6]{issa2024uniqueness}, when $g$ is not assumed to be differentiable with Lipschitz gradient there may be several Lipschitz functions satisfying \eqref{i.initial}, \eqref{i.aesol} and \eqref{i.semiconcave} simultaneously. We also point out that condition \eqref{i.semiconcave} is pivotal for the validity of Theorem~\ref{t.mainresult} as illustrated in the following example. 
\begin{example}
    Consider
        \begin{e*} 
        \begin{cases}
            \partial_t g + (\partial_x g)^2 = 0 &\text{ on } (0,+\infty) \times \R \\
            g(0,\cdot) = 0                      &\text{ on } \R.
        \end{cases}
    \end{e*}
    The viscosity solution of the equation above is the smooth function $g = 0$. The function
    \begin{e*}
        f(t,x) =    \begin{cases}
                        |x|-t &\text{ if } |x| \leq t \\
                        0      &\text{ otherwise},
                    \end{cases}
    \end{e*}
    is Lipschitz and satisfies conditions \eqref{i.initial} and \eqref{i.aesol} in Theorem~\ref{t.mainresult}, while $f \neq g$.
\end{example}
\section*{Acknowledgement}
The author is indebted to Alessio Figalli and Jean-Christophe Mourrat for helping in coming up with a sketch of the content depicted in Section~\ref{s.psi=0}. A significant part of this work was conceived and written during the 2024 Saint-Flour probability summer school, the author thanks the organizing committee (Hacene Djellout, Arnaud Guillin, Boris Nectoux) for making this workshop hospitable and stimulating. 
\section{Control on the gradient of semi-concave functions}
\begin{definition}
    Let $h : \R^D \to \R$ and $c > 0$, we say that $h$ is $c$-semi-convex on $\R^D$ when for every $x,y \in \R^D$ and $\lambda \in [0,1]$, we have 
    \begin{e*}
        h(\lambda x +(1-\lambda)y) \leq \lambda h(x) + (1-\lambda)h(y) + \frac{c}{2} \lambda(1-\lambda)|x-y|^2.
    \end{e*}
    We say that $h$ is $c$-semi-concave on $\R^D$ when $-h$ is $c$-semi-convex on $\R^D$.
\end{definition}
Note that by the parallelogram identity, a function $h$ is $c$-semi-concave if and only if $\frac{c}{2}|\cdot|^2-h$ is convex. 

We let $L^\infty(\R^D,\R)$ denote the set of essentially bounded functions $\R^D \to \R$, we equip it with the essential supremum norm,
\begin{e*}
    \|h \|_\infty = \text{ess-sup}_{x \in \R^D} |h(x)|.
\end{e*}
Sometimes, we will consider bounded functions defined on a subset $A \subset \R^D$, in this case we denote by $\|\cdot \|_{\infty,A}$ the corresponding norm. Given a Lipschitz function $h : \R^D \to \R$, $h$ is differentiable almost everywhere, according to Rademacher's theorem. We let $\nabla h$ denote the almost everywhere derivative of $h$, we have $|\nabla h| \in L^\infty(\R^D,\R)$.

\begin{proposition} \label{p.grad control}
     Let $h : \R^D \to \R$ be a bounded function, assume that there exists $c > 0$ such that $h$ is $c$-semi-concave or $c$-semi-convex. Then, $h$ is Lipschitz and 
     \begin{e*}
         \| |\nabla h| \|^2_\infty  \leq 4c \| h\|_\infty. 
     \end{e*}
     More precisely, at every point of differentiability $x \in \R^D$ of $h$ we have $|\nabla h(x)|^2 \leq 4c \| h\|_\infty$.
\end{proposition} 

\begin{proof}
    Up to replacing $h$ by $-h$, we can assume without loss of generality that $h$ is $c$-semi-convex. We define $C = 2 \sqrt{\frac{\| h\|_\infty}{c}}$ and $L = 2 \sqrt{\| h\|_\infty c}$.
    
    \noindent \emph{Step 1. } We show that for every $x,y \in \R^D$, if $|x-y| \leq C$, then 
    \begin{e*}
        |h(x) - h(y)| \leq L|x-y|.
    \end{e*}

    \noindent Without loss of generality, we may assume that $x \neq y$, define $t = \frac{C}{|x-y|} \in [1,+\infty)$ and set $z = x + t(y-x)$. We have,
    \begin{e*}
        y = \frac{1}{t} z + \left(1 - \frac{1}{t}\right) x.
    \end{e*}
    Note that $t = |z-x|/|y-x|$, by $c$-semi-concavity of $h$, we have 
    \begin{e*}
        h(y) - h(x) \leq \frac{1}{t} \left( h(z) - h(x) \right) + \frac{c}{2} \frac{1}{t} \left(1 - \frac{1}{t} \right)|z-x|^2.
    \end{e*}
    In addition,
    \begin{align*}
        \frac{1}{t} \left(1 - \frac{1}{t} \right)|z-x|^2 &= \frac{1}{t} \left(1 - \frac{1}{t} \right) t^2|y-x|^2 \\
                                                         &= (t-1)|y-x|^2 \\
                                                         &\leq t|y-x|^2 \\
                                                         &=C|x-y|.
    \end{align*}
    So,
    \begin{e*}
        h(y) - h(x) \leq \frac{2\|h\|_\infty}{t} + \frac{c}{2}C|y-x| = L|y-x|.
    \end{e*}
    Finally, since $x$ and $y$ play symmetric roles, we have 
    \begin{e*}
        |h(y) - h(x)| \leq L|y-x|.
    \end{e*}

    \noindent \emph{Step 2.} We show that $h$ is $L$-Lipschitz. 

    \noindent Fix $x,y \in \R^D$ and let $n \in \N^*$ be large enough so that $\frac{|y-x|}{n} \leq C$. For $k \in \{0,\dots,n\}$, define
    \begin{e*}
        x_k = x + \frac{k}{n}(y-x).
    \end{e*}
    We have $|x_{k+1}-x_k| \leq C$, so according to Step 1 the following holds,
    \begin{align*}
        |h(y) - h(x)| &\leq \sum_{k = 0}^{n-1} |h(x_{k+1}) -h(x_k)| \\
                      &\leq \sum_{k = 0}^{n-1} L|x_{k+1} -x_k| \\
                      &= L|x-y|.
    \end{align*}

    \noindent \emph{Step 3.} Conclusion.

    \noindent According to Step 2, $h$ is $L$-Lipschitz on $\R^D$. It follows from Rademacher's theorem that $h$ is differentiable almost everywhere. Let $x \in \R^D$ be a point of differentiability of $h$. For every $\varepsilon > 0$ and $v \in \R^D \setminus \{0\}$, we have 
    \begin{e*}
        \frac{|h(x+\varepsilon v) - h(x)|}{\varepsilon} \leq L|v|.
    \end{e*}
    Letting $\varepsilon \to 0$ and taking the supremum over $v$, it follows that $|\nabla h(x)| \leq L$. Finally, taking the essential supremum over $x \in \R^D$, we obtain 
     \begin{e*}
         \| |\nabla h| \|_\infty  \leq L = 2 \sqrt{c \| h\|_\infty}. \qedhere
     \end{e*}
\end{proof}

\section{The case of null initial condition} \label{s.psi=0}

In this section, we prove Theorem~\ref{t.mainresult} under the additional assumption that $\psi = 0$. In this case, the viscosity solution of \eqref{e.HJ} does not depend on $x$ and satisfies $g(t,x) = tH(0)$. This allows us to give a simple proof of Theorem~\ref{t.mainresult} which we use as a guideline for the proof with general $\psi$ in Section~\ref{s.proof}. This simple proof relies on the fact that thanks to a simple reparametrization, we can assume that $H(0) = 0$ and $\nabla H(0) = 0$.

\begin{proof}[Proof of Theorem~\ref{t.mainresult} when $\psi = 0$]
    Let 
    \begin{e*}
        v(t,x) = f(t,x-t\nabla H(0)) - tH(0),
    \end{e*}
    we have for almost every $(t,x) \in (0,T) \times \R^D$, 
    \begin{e*} 
        \partial_t v - \left( H(\nabla v) - H(0) - \nabla H(0) \cdot \nabla v\right)= 0.
    \end{e*}    
    Therefore, replacing $f$ by $v$ if needed, we can assume without loss of generality that $H(0) = 0$ and $\nabla H(0) = 0$. Under this assumption, we have $g = 0$ and we wish to show that $f = 0$. 
    
    For every $\ell \geq 0$, there exists a constant $c(\ell) > 0$ such that for every $p \in \R^D$ satisfying $|p| \leq \ell$, we have 
    \begin{e*}
        |H(p)| \leq c(\ell)|p|^2.
    \end{e*}
    Since $f$ is jointly Lipschitz on $[0,T] \times \R^D$, there exists $\ell \geq 0$ such that for every $(t,x) \in [0,T] \times \R^D$,
    \begin{e*}
        |f(t,x) - 0| \leq \ell t.
    \end{e*}
    Hence, $f(t,\cdot)$ is bounded and satisfies the hypotheses of Proposition~\ref{p.grad control}. In particular, taking $c = c(\ell)$ for almost every $(t,x) \in [0,T] \times \R^D$, we have 
    \begin{align*}
        |f(t,x)| &\leq \int_0^t |\partial_t f(s,x)|\d s \\
                 &\leq \int_0^t |H(\nabla f(s,x))| \d s \\
                 &\leq c \int_0^t |\nabla f(s,x)|^2 \d s \\
                 &\leq c\int_0^t \||\nabla f(s,\cdot)|\|_\infty^2 \d s \\    
                 &\leq c'\int_0^t \| f(s,\cdot)\|_\infty \d s,
    \end{align*}
    where $c' > 0$ is some constant depending on $\ell$ and the constant of \eqref{i.semiconcave}. Taking the essential supremum over $x \in \R^D$, we discover that for almost every $t \in [0,T)$ we have 
    \begin{e*}
        \| f(t,\cdot) \|_\infty \leq c'(\ell) \int_0^t \| f(s,\cdot)\|_\infty \d s.
    \end{e*}
    Finally, by Lipschitz continuity, the previous display actually holds for every $t \in [0,T)$. Applying Gronwall's lemma \cite[Theorem~2.1]{chicone2024ode}, we deduce that for every $t \in [0,T]$, $ \| f(t,\cdot) \|_\infty = 0$ and so $f = 0$.
\end{proof}

\section{Characteristic curves} \label{s.char}

In this section, we consider $g: [0,T] \times \R^D \to \R$ the viscosity solution of 
\begin{e*}
\begin{cases}
    \partial_t g - H(\nabla g) = 0 &\textrm{ on } (0,T) \times \R^D \\
    g(0,\cdot) = \psi              &\textrm{ on } \R^D.
\end{cases}
\end{e*}
We assume that $g$ satisfies the regularity assumption of Theorem~\ref{t.mainresult}. That is, $g$ is differentiable on $(0,T) \times \R^D$, and there exists $L \geq 0$ such that for every $t \in [0,T)$ and $x,y \in \R^D$,
\begin{e*}
    |\nabla g(t,x) - \nabla g(t,y)| \leq L|x-y|. 
\end{e*}
Under this assumption, we build the characteristic curves associated to $g$. Let $\msf b = - \nabla H(\nabla g)$, by the Picard-Lindelöf theorem, for every $x \in \R^D$, the ordinary differential equation 
    \begin{e*}
        \begin{cases}
             \dot{\varphi}(t) = - \msf b(t,\varphi(t)) \\
             \varphi(0) = x,
        \end{cases}
    \end{e*}
admits a unique strong solution on $[0,T)$, we denote it by $t \mapsto X(t,x)$. For every $t \in [0,T)$, we define $X^t = X(t,\cdot)$.
\begin{proposition} \label{p.bilipschitz flow}
    There exists a constant $c > 0$ such that for every $t \in [0,T)$, the map $X^t : \R^D \to \R^D$ is bijective and $e^{ct}$-Lipschitz. Furthermore, the inverse map $(X^t)^{-1}$ is $e^{ct}$-Lipschitz. 
\end{proposition}
\begin{proof} We let $c > 0$ be such that for every $t \in [0,T)$ and $x,y \in \R^D$, we have 
    \begin{e*}
       | \msf b(t,x) - \msf b(t,y) | \leq c|x-y|.
    \end{e*}

    \noindent \emph{Step 1.} We show that for every $x,y \in \R^D$, we have 
    \begin{e*}
        e^{-\frac{ct}{2}}|x-y| \leq |X^t(x) - X^t(y)| \leq e^{\frac{ct}{2}}|x-y|.
    \end{e*}

    \noindent Fix $x,y \in \R^D$, we have 
    \begin{align*}
        \frac{\d}{\d t} |X^t(x) - X^t(y)|^2 &= 2 \left( \dot{X}^t(x) - \dot{X}^t(y) \right) \cdot \left( {X}^t(x) - {X}^t(y) \right) \\
                                         &= 2 \big( \msf b(t,{X}^t(x)) - \msf b(t,{X}^t(y)) \big) \cdot \left( {X}^t(x) - {X}^t(y) \right).
    \end{align*}
    By definition of $c$, it follows from the Cauchy-Schwarz inequality that  
    \begin{e*}
        - 2c|X^t(x) - X^t(y)|^2  \leq \frac{\d}{\d t} |X^t(x) - X^t(y)|^2 \leq  2c|X^t(x) - X^t(y)|^2.
    \end{e*}
    Finally, applying Gronwall's lemma \cite[Theorem~2.1]{chicone2024ode}, we deduce that 
    \begin{e*}
        e^{-ct}|x-y| \leq |X^t(x) - X^t(y)| \leq e^{ct}|x-y|.
    \end{e*}
    %
    \noindent \emph{Step 2.} We show that for every $t \in [0,T)$, $X^t : \R^D \to \R^D$ is bijective.

    \noindent Let $t \in [0,T)$, it follows from Step 1 that $X^t$ is injective. We have for every $x \in \R^D$,
    \begin{e*}
        |X(t,x) - x| \leq t \| \msf b \|_{\infty, [0,t] \times \R^D}.
    \end{e*}
    Fix $y \in \R^D$, the previous display implies that the closed convex envelope of the set $\{y + x - X^t(x) \big| x \in \R^D\}$ is convex and compact. We let 
    \begin{e*}
        \mfk C = \overline{\text{conv}} \{y + x - X^t(x) \big| x \in \R^D\}.
    \end{e*}
    According to Step 1, the map 
    \begin{e*}
      \Phi : \begin{cases}
            \mfk C \longrightarrow  \R^D  \\
            x \longmapsto y + x - X^t(x)
        \end{cases}
    \end{e*}
    is continuous. By definition of $\mfk C$, we have $\Phi(\mfk C) \subset \{y + x - X^t(x) \big| x \in \R^D\} \subset \mfk C$. Since $\mfk C$ is convex and compact, by Brouwer's fixed point theorem, there exists $x^* \in \mfk C$ such that $x^*  = \Phi(x^*)$. In particular, we have $y = X^t(x^*)$ and $X^t$ is surjective. 

    \noindent \emph{Step 3.} Conclusion.

    \noindent Let $t \in [0,T)$, we have shown in Step 1 that $X^t$ was $e^{ct}$-Lipschitz and in Step~2 that $X^t$ was bijective. In addition, the inequality 
    \begin{e*}
        e^{ct}|x-y| \leq |X^t(x) - X^t(y)|
    \end{e*}
    is equivalent to 
    \begin{e*}
        |(X^t)^{-1}(x)-(X^t)^{-1}(y)| \leq e^{ct}|x-y|.
    \end{e*}
    Thus $(X^t)^{-1}$ is $e^{ct}$-Lipschitz.
\end{proof}

\begin{proposition}
    For every $t \in [0,T)$ and every Lebesgue negligible set $A \subset \R^D$, the sets $X^t(A)$ and $(X^t)^{-1}(A)$ are also Lebesgue negligible.
\end{proposition}

\begin{proof}
    We prove that $(X^t)^{-1}(A)$ is Lebesgue negligible only using the fact that $(X^t)^{-1}$ is Lipschitz, a similar proof would yield that $X^t(A)$ is negligible. We let $\mcl L^n$ denote the $n$-dimensional Lebesgue measure. We say that a subset of the form $Q(y_0,r) = y_0 + [-r,r]^D$ with $y_0 \in \R^D$ and $r \geq 0$ is a cube. Note that by translation invariance of the Lebesgue measure, $\mcl L^D(Q(y_0,r))$ does not depend on $y_0$. We let $c > 0$ denote the constant appearing in Proposition~\ref{p.bilipschitz flow}.

    \noindent \emph{Step 1.} We show that for every $t \in [0,T)$ and every cube $Q \subset \R^D$, we have 
    \begin{e*}
        \mcl L^D \left( (X^t)^{-1}(Q)\right) \leq \left(\sqrt{D}e^{ct}\right)^D \mcl L^D(Q).
    \end{e*}

    \noindent  Let $y_0 \in \R^D$ and $r \geq 0$ such that $Q = Q(y_0,r)$. We let $|\cdot|_\infty$ denote the sup norm on $\R^D$. According to Proposition~\ref{p.bilipschitz flow}, there exists $c > 0$ such that $X^t$ is $e^{ct}$-Lipschitz with respect to $|\cdot|$. We have for every $y \in Q$,
    \begin{align*}
        |(X^t)^{-1}(y)- (X^t)^{-1}(y_0)|_\infty &\leq |(X^t)^{-1}(y)- (X^t)^{-1}(y_0)| \\
                                                &\leq e^{ct} |y-y_0| \\
                                                &\leq e^{ct} \sqrt{D} |y-y_0|_\infty \\
                                                &\leq e^{ct} \sqrt{D}r.
    \end{align*}
    Thus, $(X^t)^{-1}(Q) \subset Q((X^t)^{-1}(y_0),e^{ct} \sqrt{D}r)$. It follows that 
    \begin{align*}
        \mcl L^D \left( (X^t)^{-1}(Q) \right) &\leq \mcl L^D \left( Q((X^t)^{-1}(y_0),e^{ct} \sqrt{D}r) \right) \\
                                              &= \left(e^{ct} \sqrt{D}\right)^D \mcl L^D(Q((X^t)^{-1}(y_0),r)) \\
                                              &= \left(e^{ct} \sqrt{D}\right)^D \mcl L^D(Q(y_0,r)).
    \end{align*}

    \noindent \emph{Step 2.} We show that for every $t \in [0,T)$ and every $\mcl L^D$-negligible set $A \subset \R^D$, the set $(X^t)^{-1}(A)$ is also $\mcl L^D$-negligible.

    \noindent By definition of the Lebesgue measure, for every $\varepsilon > 0$ there exists a sequence of cubes $(Q_n)_{n \geq 1}$ such that
    \begin{e*}
        A \subset \bigcup_{n = 1}^\infty Q_n, 
    \end{e*}
    and $\sum_{n = 1}^\infty \mcl L^D(Q_n) \leq \varepsilon$. According to Step 1, for every $n \geq 1$ we have 
    \begin{e*}
        \mcl L^D((X^t)^{-1}(Q_n)) \leq (\sqrt{D}e^{ct})^D \mcl L^D(Q_n).
    \end{e*}
    Therefore,
    \begin{align*}
        \mcl L^D((X^t)^{-1}(A)) &\leq \mcl L^D \left( \bigcup_{n = 1}^\infty (X^t)^{-1}(Q_n)  \right) \\
                                &\leq \sum_{n = 1}^\infty  \mcl L^D \left( (X^t)^{-1}(Q_n)  \right) \\
                                &\leq \sum_{n = 1}^\infty (\sqrt{D}e^{ct})^D \mcl L^D(Q_n) \\
                                & \leq (\sqrt{D}e^{ct})^D \varepsilon.
    \end{align*}
    Letting $\varepsilon \to 0$, we obtain that $\mcl L^D((X^t)^{-1}(A)) = 0$.
\end{proof}

We define the inverse flow map $Y : [0,T) \times \R^D \to [0,T) \times \R^D$ by 
\begin{e*}
    Y(t,x) = (t,(X^t)^{-1}(x)).
\end{e*}
According to Proposition~\ref{p.bilipschitz flow}, $Y$ is bijective, we let $Y^{-1}$ denote its inverse.

\begin{proposition} \label{p. space-time flow is absolutely continuous}
    For every Lebesgue negligible set $N \subset [0,T) \times \R^D$, the sets $Y^{-1}(N)$ and $Y(N)$ are Lebesgue negligible.
\end{proposition}

\begin{proof}

    Let $N \subset [0,T) \times \R^D$ such that $\mcl L^{D+1}(N) = 0$. We prove that $Y^{-1}(N)$ is Lebesgue negligible, a similar proof would also yield that $Y(N)$ is Lebesgue negligible. For every $t \in [0,T)$, define 
    \begin{e*}
        N_t = \{ x \in \R^D \big| (t,x) \in N \}.
    \end{e*}
    We have 
    \begin{e*}
        \int_0^T \mcl L^D(N_t) \d t = \mcl L^{D+1}(N) = 0,
    \end{e*}
    so for almost all $t \in [0,T)$, $\mcl L^D(N_t) = 0$. It follows from Step 2 that for almost all $t \in [0,T)$, $\mcl L^D((X^t)^{-1}(N_t)) = 0$ and thus 
    \begin{align*}
        \mcl L^{D + 1} (Y^{-1}(N)) &= \int_{[0,T) \times \R^D} \mathbf{1}_{Y^{-1}(N)}(t,x) \d x \d t  \\
                              &= \int_0^T \int_{\R^D} \mathbf{1}_{N}(t,(X^t)^{-1}(x)) \d x \d t \\
                              &= \int_0^T \int_{\R^D} \mathbf{1}_{(X^t)^{-1}(N_t)}(x) \d x \d t \\
                              &= \int_0^T \mcl L^D((X^t)^{-1}(N_t)) \d t \\
                              &= 0. \qedhere                      
    \end{align*}
\end{proof}

For the sake of completeness, we finish this section by proving the following proposition, which states that when $\nabla g$ is locally jointly Lipschitz, $\nabla g$ is constant along the flow of solution $Y^{-1}$ and that the characteristic curves are in fact straight lines. We will not use this result in the proof of Theorem~\ref{t.mainresult}, but it gives us a hint on how to generalize the construction of the function $v$ appearing in Section~\ref{s.psi=0}.

\begin{proposition} \label{p.constant gradient along char}
   Assume that for every compact subset $K \subset [0,T) \times \R^D$, $\nabla g$ is Lipschitz on $K$. Then, for every $x \in \R^D$, the function $t \mapsto \nabla g(t,X(t,x))$ is constant on $[0,T)$ and for every $t \in [0,T)$, we have 
    \begin{e*}
        X(t,x) = x - t \nabla H(\nabla \psi(x)).
    \end{e*}
\end{proposition}

\begin{proof}
    The function $\nabla g$ is locally Lipschitz on $(0,T) \times \R^D$ so it is differentiable almost everywhere on $(0,T) \times \R^D$, differentiating \eqref{e.HJ} with respect to $x$, we have for almost all $(t,x) \in (0,T) \times \R^D$,
    \begin{e*}
        \partial_t \nabla g(t,x) = \nabla^2 g(t,x) \nabla H(\nabla g(t,x)).
    \end{e*}
    Furthermore, according to Proposition~\ref{p. space-time flow is absolutely continuous}, for almost every $(t,x) \in (0,T) \times \R^D$, $\nabla g$ is differentiable at $(t,X^t(x))$. Using the previous display, we discover that for almost all $(t,x) \in (0,T) \times \R^D$,
    \begin{align*}
        \frac{\d}{\d t} \nabla g(t,X^t(x)) &= \partial_t \nabla g(t,X^t(x)) + \nabla^2g(t,X^t(x)) \dot{X}(t,x) \\
                                           &= \left( \partial_t \nabla g - \nabla^2g \nabla H(\nabla g)\right)(t,X^t(x)) \\
                                           &= 0.
    \end{align*}
    The function $t \mapsto \nabla g(t,X^t(x))$ is absolutely continuous as the composition of two Lipschitz continuous functions. Thus, we can integrate the previous display with respect to $t$ and obtain that for almost all $x \in \R^D$, we have for all $t \in (0,T)$, $\nabla g(t,X^t(x)) = \nabla \psi(x)$. Finally, since $x \mapsto \nabla g(t,X^t(x))$ and $x \mapsto \nabla \psi(x)$ are continuous on $\R^D$ we have as desired for every $(t,x) \in (0,T) \times \R^D$,
    \begin{e*}
        \nabla g(t,X^t(x)) = \nabla \psi(x).
    \end{e*}
    In particular, since we have $\dot{X}^t(x) = - \nabla H(\nabla g(t,X^t(x))$, we discover that $\dot{X}(t,x) = -\nabla H(\nabla \psi(x))$ and thus 
    \begin{e*}
        X(t,x) = x - t \nabla H(\nabla \psi(x)).
    \end{e*}
\end{proof}

As a consequence of Proposition~\ref{p.constant gradient along char}, we see that the function $v$ of Section~\ref{s.psi=0} is in fact the function $(t,x) \mapsto (f-g)(t,X^t(x))$. This fact will be our guiding rail to prove Theorem~\ref{t.mainresult}.

\section{Proof of the main result} \label{s.proof}
 
In this section, we adapt the argument of Section~\ref{s.psi=0} using the results of Section~\ref{s.char} to prove Theorem~\ref{t.mainresult}. In what follows, we fix $H : \R^D \to \R$ a $\mcl C^2$ function and $\psi : \R^D \to \R$ a Lipschitz and differentiable function with Lipschitz gradient. We also assume that $g$ the viscosity solution of \eqref{e.HJ} is differentiable, and that there exists $L \in \R$ such that for every $t \in [0,T)$ and every $x,y \in \R^D$,
\begin{e*}
    |\nabla g(t,x) - \nabla g(t,y)| \leq L|x-y|. 
\end{e*}
%


\begin{lemma} \label{l. control on the difference}
Assume the hypotheses of Theorem~\ref{t.mainresult} and let $u = f-g$. There exists a constant $c > 0$ such that for almost all $x \in \R^D$, for all $t \in [0,T)$,
\begin{e*}
    |u(t,X^t(x))| \leq c\int_0^t |\nabla u(s,X^s(x))|^2  \d s.
\end{e*}
\end{lemma}

\begin{proof} 
    Given a matrix $A \in \R^{D \times D}$, we define $|A|_{\text{op}} = \sup_{|x| \leq 1} x \cdot Ax$. Let $L \geq 0$ be such that for every $t \in [0,T)$, $f(t,\cdot)$ and $g(t,\cdot)$ are $L$-Lipschitz, we set
    \begin{e*} 
        c = \frac{1}{2} \sup_{ |r| \leq 2L} |\nabla^2 H(r)|_{\text{op}}.
    \end{e*}
    For every $p,q \in \R^D$ such that $|p|,|q| \leq L$, we have 
    \begin{e*}
        |H(p+q) - H(p) - q \cdot \nabla H(q)| \leq c |q|^2.
    \end{e*}
    By Rademacher's theorem, $f$ is differentiable almost everywhere on $(0,T) \times \R^D$. Thus, according to Proposition~\ref{p. space-time flow is absolutely continuous}, we have that for almost all $(t,x) \in (0,T) \times \R^D$, $f$ is differentiable at $(t,X^t(x))$. Fix $x \in \R^D$ such that, for almost all $t \in (0,T)$, $f$ is differentiable at $(t,X^t(x))$. At every point of differentiability $s \in (0,T)$, we have 
    \begin{align*}
        \frac{\d}{\d s} u(s,X^s(x)) &= \partial_t u(s,X^s(x)) + \dot{X}(s,x) \cdot \nabla u(s,X^s(x)) \\
                                    &= \left(H(\nabla f) - H(\nabla g) - \nabla H(\nabla g) \cdot \nabla u \right)(s,X^s(x)) \\
                                    &= \left( H(\nabla g + \nabla u) - H(\nabla g) - \nabla H(\nabla g) \cdot \nabla u \right)(s,X^s(x)).
    \end{align*}
    The function $t \mapsto u(t,X^t(x))$ is Lipschitz so for every $t \in [0,T)$, we have 
    \begin{e*}
        u(t,X^t(x)) = \int_0^t \frac{\d}{\d s} u(s,X^t(x)) \d s.
    \end{e*}
    Hence, for all $t \in [0,T)$, 
    \begin{align*}
        |u(t,X^t(x))| &\leq \int_0^t \left| \frac{\d}{\d s} u(s,X^t(x)) \right| \d s \\
                      &= \int_0^t |\left( H(\nabla u + \nabla g) - H(\nabla g) - \nabla H(\nabla g) \cdot \nabla u \right)(s,X^s(x))| \d s \\
                      &\leq c\int_0^t|\nabla u(s,X^s(x))|^2 \d s \\ 
    \end{align*}
    This is the desired result.    
\end{proof}

\begin{proof}[Proof of Theorem~\ref{t.mainresult}]
    By definition, we have that $g$ the viscosity solution of \eqref{e.HJ} is a Lipschitz function on $[0,T] \times \R^D$ satisfying $g(0,\cdot) = \psi$. According to \cite[Section~10~Theorem~1]{evans}, for all $(t,x) \in (0,T) \times \R^D$, we have 
    \begin{e*}
        \partial_t g(t,x) - H(\nabla g(t,x)) = 0. 
    \end{e*}
    In addition, since for every $t \in [0,T)$, $\nabla g(t,\cdot)$ is $L$-Lipschitz, it follows that for every $x,y \in \R^D$,
    \begin{e*}
        (\nabla g(t,x) - \nabla g(t,y) ) \cdot (x-y) \leq L|x-y|^2.
    \end{e*}
    The previous display implies that $g(t,\cdot)$ is $L$-semi-concave and thus $g$ is a Lipschitz function satisfying \eqref{i.initial}, \eqref{i.aesol} and \eqref{i.semiconcave} in Theorem~\ref{t.mainresult}. Let us now prove that $g$ is the unique Lipschitz function satisfying \eqref{i.initial}, \eqref{i.aesol} and \eqref{i.semiconcave} in Theorem~\ref{t.mainresult}. To do so, we will prove that the function $u = f-g$ is the null function. Let us mention that a similar argument than the one that yielded the $L$-semi-concavity of $g(t,\cdot)$ can be used to prove that $g(t,\cdot)$ is $L$-semi-convex, this means in particular that the function $u(t,\cdot)$ is semi-concave.
    
    Let $c > 0$ be the constant appearing in Lemma~\ref{l. control on the difference}. According to Proposition~\ref{p. space-time flow is absolutely continuous}, for almost all $(t,y) \in \R^D$, Lemma~\ref{l. control on the difference} holds at $(t,(X^t)^{-1}(y))$, thus 
    \begin{e*}
        |u(t,y)| \leq c\int_0^t |\nabla u(s,X^s((X^t)^{-1}(y))|^2 \d s.
    \end{e*}
    Taking the essential supremum over $y \in \R^D$, we obtain that for all $t \in [0,T)$, 
    \begin{e*}
        \|u(t,\cdot)\|_{\infty} \leq c\int_0^t \| |\nabla u(s,\cdot)| \|^2_\infty \d s.
    \end{e*}
    Finally, appealing to Proposition~\ref{p.grad control} we obtain for all $t \in [0,T)$,
    \begin{e*}
        \|u(t,\cdot)\|_{\infty} \leq c'\int_0^t \| u(s,\cdot) \|_\infty \d s,
    \end{e*}
    for some constant $c' > 0$. It then follows by Gronwall's lemma \cite[Theorem~2.1]{chicone2024ode} that for every $t \in [0,T)$, $\|u(t,\cdot)\|_{\infty} = 0$ and thus $f = g$ on $[0,T] \times \R^D$.
\end{proof}

\newpage

\small
\bibliographystyle{plain}
\bibliography{ref}

\end{document}